\documentclass[a4paper,11pt]{article}
\usepackage{amsmath}
\usepackage{amsthm}
\usepackage{amssymb}
\usepackage{url}
\usepackage{graphicx}
\theoremstyle{definition}

\newtheorem{theorem}{Theorem}[section]
\newtheorem{definition}[theorem]{Definition}
\newtheorem{proposition}[theorem]{Proposition}
\newtheorem{lemma}[theorem]{Lemma}
\newtheorem{corollary}[theorem]{Corollary}

\newtheorem*{remark}{Remark}
\newtheorem*{acknowledgment}{Acknowledgment}

\newcommand{\Natural}{\mathbb{N}}
\newcommand{\Nonnegative}{\mathbb{N}_0}
\newcommand{\Integer}{\mathbb{Z}}
\newcommand{\Rational}{\mathbb{Q}}
\newcommand{\Real}{\mathbb{R}}
\newcommand{\Complex}{\mathbb{C}}
\newcommand{\abs}[1]{\left\lvert #1 \right\rvert}
\newcommand{\norm}[1]{\left\lVert #1 \right\rVert}
\newcommand{\floor}[1]{\left\lfloor #1 \right\rfloor}
\newcommand{\ceil}[1]{\left\lceil #1 \right\rceil}

\newcommand{\set}[1]{\left\{ #1 \right\}}
\newcommand{\tset}[1]{\{ #1 \}}
\newcommand{\setcond}[2]{\left\{ #1 \,;\, #2 \right\}}
\newcommand{\tsetcond}[2]{\{ #1 \,;\, #2 \}}

\newcommand{\T}{\mathcal{T}}

\newcommand{\TLattice}{\mathcal{T}_{\mathbb{N}_0}}

\newcommand{\Pas}[1]{\mathrm{Pas} \left( {#1} \right)}
\newcommand{\Paszeta}[4]{Z_{#1} \left( {#2},{#3};{#4} \right)}

\newcommand{\counting}[3]{\phi_{#1} \left( #2 ; #3 \right)}
\newcommand{\summatory}[3]{N_{#1} \left( #2 ; #3 \right)}
\newcommand{\coefficient}[3]{\psi_{#1} \left( #2 ; #3 \right)}
\newcommand{\ct}[2]{\phi_{#1} \left( #2 \right)}
\newcommand{\sm}[2]{N_{#1} \left( #2 \right)}
\newcommand{\ce}[2]{\psi_{#1} \left( #2 \right)}
\newcommand{\smX}[2]{N_{#1}^{*} \left( #2 \right)}

\numberwithin{equation}{section}

\title{\textbf{Non-real poles on the axis of absolute convergence of the zeta functions associated to Pascal's triangle modulo a prime}}
\author{IKKAI, Tomohiro}
\date{}

\begin{document}

\maketitle

\begin{abstract}
Picking binomial coefficients which cannot be divided by a given prime from Pascal's triangle, 
we find that they form a set with self-similarity.
Essouabri studied on a class of meromorphic functions associated to the above set.
These functions are related to fractal geometry 
and it is a problem whether such a function has a non-real pole on its axis of absolute convergence.

Essouabri gave a proof of existence of such a non-real pole in the simplest case.
The keys of his proof are Stein's and Wilson's estimates on how fast the points multiply in Pascal's triangle modulo a prime.
This article will give an extension of Essouabri's result to some cases 
with certain ways to count the points in Pascal's triangle modulo a prime which are different from the traditional one.
\end{abstract}

\section{Introduction} \label{secIntro}

Let $\Natural$ be the set of positive integers, $\Nonnegative = \Natural \cup \tset{0}$, 
$\Integer$ be the ring of rational integers, $\Rational$ be the field of rational numbers,
$\Real$ be the field of real numbers, $\Complex$ be the field of complex numbers, respectively.

The set 
\begin{equation}
\TLattice = \setcond{(m, n) \in \Nonnegative \times \Nonnegative}{m \geq n} \notag
\end{equation}
can be regarded as Pascal's triangle in the sense that each point $(m, n) \in \TLattice$ corresponds to a binomial coefficient 
$\binom{m}{n} = m!/(n! (m-n)!)$.
In this paper, we prefer to make arguments on Pascal's triangle in the form $\TLattice$ 
instead of the familiar form of an equilateral triangle.

We consider the set
\begin{equation*}
\Pas{p} = \setcond{(m, n) \in \TLattice}{\binom{m}{n} \not\equiv 0 \pmod{p}}
\end{equation*}
for a prime number $p$, of which the distribution of the points has been studied for a long time.
It is classically known that the set $\Pas{p}$ has ``self-similarity'' as seen in Figure 1 and Figure 2, 
where the points with a filled mark belong to $\Pas{p}$.
Kummer \cite{Kum1852} first gave a criterion when a power of $p$ divides a binomial coefficient 
in terms of the expansion of non-negative integers in the base $p$.
We can determine which point $(m, n) \in \TLattice$ belongs to $\Pas{p}$ by Kummer's criterion, 
or by Lucas' formula on binomial coefficients appearing in Section \ref{secMainB}, which makes the argument simpler.
In fact, such an arithmetic property causes $\Pas{p}$ to have ``self-similarity.'' 


\begin{figure}[t]
\begin{minipage}{0.5\hsize}
\centering
\includegraphics[width=6.0cm,clip]{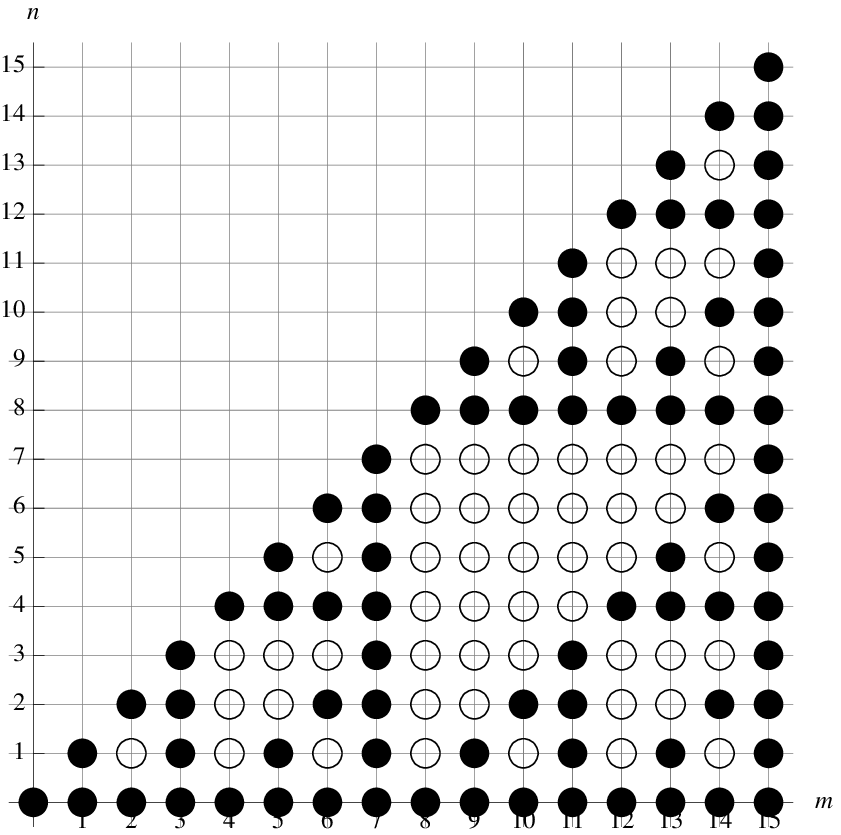}
\caption{$0 \leq m < 16$ in $\Pas{2}$}
\label{figPas2}
\end{minipage}
\begin{minipage}{0.5\hsize}
\centering
\includegraphics[width=6.0cm,clip]{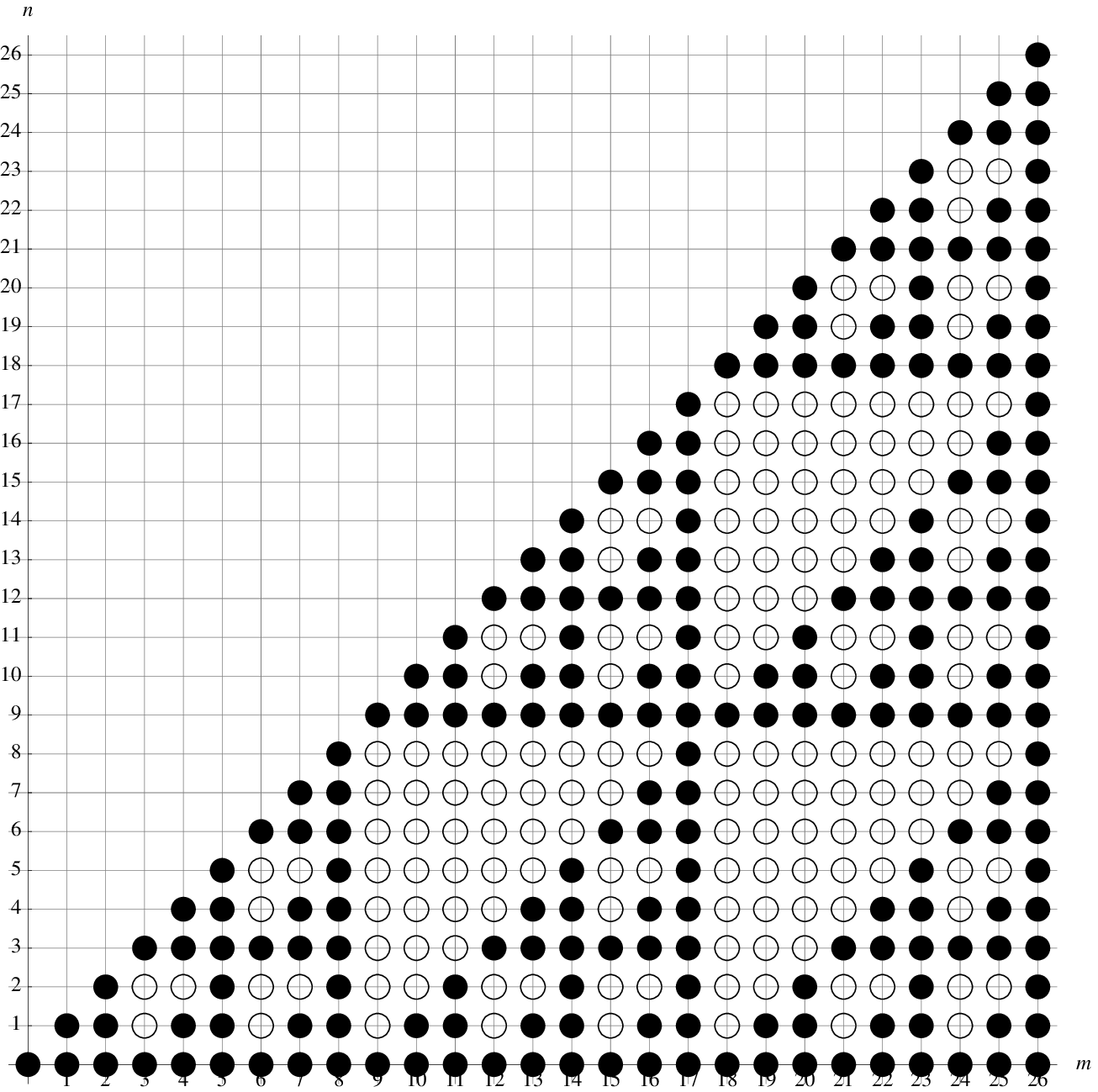}
\caption{$0 \leq m < 27$ in $\Pas{3}$}
\label{figPas3}
\end{minipage}
\end{figure}

Essouabri \cite{Ess05} introduced the zeta function associated to $\Pas{p}$, defined by 
\begin{equation*}
\Paszeta{p}{P}{Q}{s} = \sum_{\begin{subarray}{c} (m,n) \in \Pas{p} \\ P(m,n) \neq 0 \end{subarray}} \frac{Q(m,n)}{P(m,n)^{s/\deg P}},
\end{equation*}
where $s = \sigma + it \in \Complex$ with the sufficiently large real part $\sigma$, the imaginary part $t$ 
and $P$, $Q$ are two-variable polynomials with real coefficients.
Moreover, $P$ is required to be ``$\T$-elliptic'' to ensure the convergence; 
$P \in \Real[X,Y]$ is said to be $\T$-elliptic if $P$ is non-constant and its highest degree part $P^{*}$ 
(which means that $P^{*}$ is homogeneous and $\deg (P-P^{*}) < \deg P$) 
satisfies $P^{*}(x,y) > 0$ for every $(x,y) \in \Real^2 \setminus \tset{(0,0)}$ with $x \geq y \geq 0$. 

Essouabri showed several analytic properties of $\Paszeta{p}{P}{Q}{s}$ 
including the meromorphic continuation, the location of possible poles and the abscissa of absolute convergence.
In particular, in the case $P(X,Y) = X$, he proved an intriguing fact as follows. 

\begin{theorem}[Essouabri, 2005 \cite{Ess05}] \label{thEss}
Let $\theta_p = \log \frac{p(p+1)}{2} / \log p$. 
Then the meromorphic function on the whole complex plane $\Paszeta{p}{X}{1}{s}$ has at least two non-real poles 
on its axis of absolute convergence $\tset{\sigma = \theta_p}$.
(Let ``the axis of absolute convergence'' represent the vertical line in the complex plane through the abscissa of absolute convergence.)
\end{theorem}

Indeed, we only have to verify that $\Paszeta{p}{X}{1}{s}$ has one desired pole 
since the second one can be obtained by using the Schwarz reflection principle.

Recalling Essouabri's proof of Theorem \ref{thEss}, 
we note that a key role is played by some estimates for how fast the points in $\Pas{p}$ multiply. 
More precisely, certain estimates for an arithmetic function $\smX{p}{u} = \# \tsetcond{(m,n) \in \Pas{p}}{m<u}$ 
defined for $u \in \Natural$ are significant in his proof.
Actually,
\begin{equation*}
\Paszeta{p}{X}{1}{s} = \int_{1^{-}}^{\infty} u^{-s} \, d\smX{p}{u}
\end{equation*}
holds for $s \in \Complex$ at which $\Paszeta{p}{X}{1}{s}$ converges absolutely and 
we would find that the Wiener--Ikehara theorem (see \cite[Chapter 5, Corollary 1]{Hla-Sch-Tas} for example) would give 
the existence of the limit value $\lim_{u \to \infty} \smX{p}{u}/u^{\theta_p}$ if $\Paszeta{p}{X}{1}{s}$ had no non-real poles 
on the line in question.
However, this is in fact a contradiction to the known results 
on $\limsup_{u \to \infty} \smX{p}{u}/u^{\theta_p}$ and $\liminf_{u \to \infty} \smX{p}{u}/u^{\theta_p}$, 
due to Harborth \cite{Har77}, Stolarsky \cite{Sto77}, Stein \cite{Ste89} and Wilson \cite{Wil98}. 
(Details will appear in Section \ref{secPrelim}.)

The purpose of this article is to show the results parallel to Theorem \ref{thEss} in case $P(X,Y) = X+Y$ and in case $P(X,Y) = X+2Y$, $p=2$.
Those are specifically stated as follows.

\begin{theorem} \label{thGoal}
$(1)$\, For any prime $p$, 
the meromorphic function $\Paszeta{p}{X+Y}{1}{s}$ has a non-real pole (hence at least two non-real poles) on its axis of absolute convergence $\tset{\sigma = \theta_p}$.

$(2)$\, The meromorphic function $\Paszeta{2}{X+2Y}{1}{s}$ has a non-real pole (hence at least two non-real poles) on its axis of absolute convergence 
$\tset{\sigma = \theta_p}$.
\end{theorem}

It is significant to consider $\Paszeta{p}{P}{1}{s}$ in the context of fractal geometry.
Actually, it can be realized as the geometric zeta function of a certain fractal string with a scale transformation of $s$. 
Such a zeta function contains some geometric information of the fractal string in its poles, which are called the complex dimensions. 
In particular, important is the relationship between the existence of non-real poles 
on the axis of absolute convergence of the geometric zeta function and the Minkowski measurability of the corresponding fractal string, 
which has some connection with the zeros of the Riemann zeta function in the critical strip via spectrum theory.
(For details, see \cite{Lap-vFr} for example.)

We will prove Theorem \ref{thGoal} by imitating Essouabri's proof in the following sections. 
To make the same argument as his, we investigate the behavior of the functions connected with the desired poles of $\Paszeta{p}{P}{1}{s}$.
Those functions are defined as follows: 
\begin{align*}
\summatory{p}{P}{u} &= \# \tsetcond{(m,n) \in \Pas{p}}{P(m,n) < u}, \\
\coefficient{p}{P}{u} &= \summatory{p}{P}{u} / u^{\theta_p /d},
\end{align*}
where $P \in \Real[X,Y]$ is a $\T$-elliptic polynomial with degree $d \geq 1$ and $u \in \Natural$.
Of course the former is an analog to $\smX{p}{u}$ and the latter is that to $\smX{p}{u}/u^{\theta_p}$ appearing above. 

The following theorem shows the reason why we consider the function $\coefficient{p}{P}{u}$.

\begin{theorem} \label{thMainA}
Let $P \in \Real[X,Y]$ be a $\T$-elliptic polynomial with degree $d \geq 1$.
Moreover, assume that $P(x,y) \geq 0$ for every $(x,y) \in \Real^2$ with $x \geq y \geq 0$. 
($P$ shall be said to be ``$\T$-positive'' if this property is satisfied.)
Then we have $\summatory{p}{P}{u} \asymp u^{\theta_p /d}$ as $u \to \infty$. 
\end{theorem}

We give some necessary estimates for the bounded function $\coefficient{p}{P}{u}$ to ensure Theorem \ref{thGoal}.

\begin{theorem} \label{thMainB}
$(1)$\, For every prime $p$, the function $\coefficient{p}{X+Y}{u}$ fails to converge as $u \to \infty$. \\
$(2)$\, The function $\coefficient{2}{X+2Y}{u}$ fails to converge as $u \to \infty$.
\end{theorem}

We will prove that Theorem \ref{thMainB} implies Theorem \ref{thGoal} at the end of Section \ref{secPrelim} and 
the above two theorems Theorem \ref{thMainA} and Theorem \ref{thMainB} in Section \ref{secMainA} and Section \ref{secMainB}, respectively.
 
To prove the second part of Theorem \ref{thMainB}, we actually observe some arithmetic properties of $\summatory{p}{X+pY}{u}$. 
There appears an algebraic difficulty when $p \geq 3$, so that we obtain the result only in the case when $p = 2$. 
It is mentioned again in Section \ref{secMainB} why we have to restrict $p$.
The desired results in Theorem \ref{thMainB} are derived from some elementary calculations and the Gel'fond--Schneider theorem, 
which gives the affirmative answer for Hilbert's seventh problem: 
``Does $\alpha^\beta$ become transcendental when $\alpha$ and $\beta$ are algebraic over $\Rational$ 
with $\alpha \neq 0, 1$ and $\beta \notin \Rational$?''
Our calculations are out of analytic estimates unlike the former studies \cite{Ste89} and \cite{Wil98}.

In our proof of Theorem \ref{thMainB}, we can obtain certain bounds 
for $\limsup_{u \to \infty} \coefficient{p}{P}{u}$ and $\liminf_{u \to \infty} \coefficient{p}{P}{u}$ concerned with our arguments. 
Those bounds are cruder, compared with sharper results in \cite{Har77}, \cite{Sto77}, \cite{Ste89} and \cite{Wil98}, 
but they are sufficient for our present purpose.
It seems to be more difficult to improve the bounds 
since no effective formula expressing the value of $\summatory{p}{P}{u}$ is known unlike the case when $P(X,Y) = X$.

\section{Some known results} \label{secPrelim}
First, we give some notations.
\begin{itemize}
\item For $x \in \Real$, let $\floor{x}$ denote the largest integer that is not more than $x$ 
and $\ceil{x}$ denotes the smallest integer that is not less than $x$. 

\item Let $\norm{\,\cdot\,}$ denote the Euclidean norm on $\Real^2$, i.e. $\norm{(x,y)} = \sqrt{x^2 + y^2}$ for $(x,y) \in \Real^2$.

\item Let $f(x)$ and $g(x)$ be functions defined on a subset of $\Real$, e.g. $(0, \infty)$.
We write $f(x) \ll g(x) \, (x \to \infty)$ 
if there exists a positive constant $C$ such that $\abs{f(x)} \leq Cg(x)$ holds for any sufficiently large $x$. 
(It means the same as $f(x) = O(g(x)) \, (x \to \infty)$.)
We use the notation $f(x) \asymp g(x) \, (x \to \infty)$ 
if both $f(x) \ll g(x) \, (x \to \infty)$ and $g(x) \ll f(x) \, (x \to \infty)$ hold. 

\item Let $\T$, $\TLattice$ and $\T(R)$ with $R>0$ each denote the following subset of $\Real^2$; 
$\T = \tsetcond{(x,y) \in \Real^2}{x \geq y \geq 0}$, 
$\TLattice = \T \cap (\Nonnegative \times \Nonnegative)$ and 
$\T(R) = \setcond{(x,y) \in \T}{\norm{(x,y)} \geq R}$.

\item For a prime number $p$, let $I_p = \tset{0,1,\dots,p-1}$ and
$\Pas{p} = \tsetcond{(m,n) \in \TLattice}{\binom{m}{n} \not\equiv 0 \pmod p}$.
In addition, a real number $\theta_p$ is defined as $\theta_p = \log \frac{p(p+1)}{2} / \log p$, 
or $p^{\theta_p} = \frac{p(p+1)}{2}$ equivalently.

\item We shall approve the situation in which $a = \sum_{j=0}^{h} a_j p^j$ with $a_j \in I_p$, 
the expansion of $a \in \Nonnegative$ in the base a prime $p$, has its top digits $a_h, a_{h-1},\dots, a_{h-k}$ being all $0$
for some $0 \leq k \leq h$.

\item Let $s = \sigma + it \in \Complex$ consists of its real part $\sigma \in \Real$, its imaginary part $t \in \Real$ and 
the imaginary unit $i$.
\end{itemize}

We will introduce three functions which perform the leading roles in this article. 
Before giving the definitions of them, we recall a significant property of polynomials, say, being ``$\T$-elliptic.''
Incidentally, we provide a jargon ``$\T$-positive'' for convenience.

\begin{definition}
Let $P \in \Real[X,Y]$ be a non-constant polynomial with its degree $d \geq 1$ and 
$P_d \in \Real[X,Y]$ be the $d$-degree part of $P$, 
namely, the homogeneous polynomial uniquely determined by $\deg (P- P_d) < \deg P$. \\
$(1)$\, $P$ is said to be \textbf{$\T$-elliptic} if $P_d (x,y) > 0$ holds for every $(x,y) \in \T \setminus \tset{(0,0)}$. \\
$(2)$\, $P$ is said to be \textbf{$\T$-positive} if $P(x,y) \geq 0$ holds for every $(x,y) \in \T$.
\end{definition}

\begin{remark}
Being $\T$-elliptic does not necessarily imply being $\T$-positive and vice versa.
Indeed, we can easily check, for example, that the polynomial $Y$ is not $\T$-elliptic but $\T$-positive and
$X-1$ satisfies the contrary.
However, in fact, a $\T$-elliptic polynomial $P$ can be shifted by some $c>0$ as $P+c$ becomes also $\T$-positive.
This actually follows from Lemma \ref{lemEll} appearing later in Section \ref{secMainA}.
\end{remark}

Now we define three important functions in this article.

\begin{definition}
Assume that $P \in \Integer[X,Y]$ is a $\T$-elliptic and $\T$-positive polynomial with its degree $d \geq 1$.
We let $\counting{p}{P}{q}$, $\summatory{p}{P}{u}$ and $\coefficient{p}{P}{u}$ denote 
the functions in $q \in \Nonnegative$ or $u \in \Natural$ as follows: 
\begin{align*}
\counting{p}{P}{q} &= \# \setcond{(m,n) \in \Pas{p}}{P(m, n) = q}, \\
\summatory{p}{P}{u} &= \# \setcond{(m,n) \in \Pas{p}}{P(m, n) < u} = \sum_{0 \leq q < u} \counting{p}{P}{q}, \\
\coefficient{p}{P}{u} &= \frac{\summatory{p}{P}{u}}{u^{\theta_p /d}},
\end{align*}
respectively.
\end{definition}

We note that $\coefficient{p}{P}{u}$ is a bounded function, provided Theorem \ref{thMainA} is true.

\begin{remark}
If $P \in \Real[X,Y]$ is $\T$-elliptic, the curve $\set{P(x,y) = u} \subset \Real^2$ always cuts some bounded domains away from $\T$. 
This is the reason why $\counting{p}{P}{q}$ and $\summatory{p}{P}{u}$ are well-defined. 
However, being $\T$-elliptic is not the necessary condition of well-definedness of $\counting{p}{P}{q}$ and $\summatory{p}{P}{u}$. 
For example, consider the polynomial $XY+X$, not $\T$-elliptic.
\end{remark}

It has been studied how many points of $\Pas{p}$ are included in the first $u$ columns, 
the number of which is expressed by $\smX{p}{u} = \summatory{p}{X}{u}$ here. 
In particular, the asymptotic behavior of $\smX{p}{u}$ was often considered in former studies. 

It is trivially estimated that $\smX{p}{u} = O(u^2) \, (u \to \infty)$ and 
the first non-trivial estimate $\smX{p}{u} = o(u^2) \, (u \to \infty)$ was given by Fine \cite{Fin47}. 
At present it is known that the concrete order of $\smX{p}{u}$ coincides with $u^{\theta_p}$. 
In case $p=2$, such a result was earlier given by Stolarsky \cite{Sto77} than that for the other $p$'s; 
it was shown that $u^{\theta_2}/3 < \smX{2}{u} < 3u^{\theta_2}$. 

Evaluations of $\alpha_p = \limsup_{u \to \infty} \coefficient{p}{X}{u}$ and of $\beta_p = \liminf_{u \to \infty} \coefficient{p}{X}{u}$
for $p=2$ appeared in the same article, which are $1 \leq \alpha_2 \leq 1.052$, $0.72 \leq \beta_2 \leq 3^{\theta_2}/7 \, (< 0.815)$. 
Sharper evaluations of $\alpha_2$ and $\beta_2$ were given by Harborth \cite{Har77}, 
which state that $\alpha_2 = 1$ (of course being the sharpest one) and 
$0.812556 \leq \beta_2 < 0.812557$ (the exact value computed to the sixth decimal). 
While the superior limit value $\alpha_p$ for general $p$'s was explicitly computed by Stein \cite{Ste89} that $\alpha_p = 1$, 
it is too difficult to compute $\beta_p$'s exactly at present. 
However, a general evaluation of $\beta_p$ that 
\begin{equation*}
\left( 1-2^{\frac{1}{1-\theta_p}} \right)^{{\theta_p}-1} \leq \beta_p < \frac{3-\theta_p}{2(2-\theta_p)^{2-\theta_p}}
\end{equation*}
was given by Wilson \cite{Wil98}, which is useful for our purpose because of its assurance that $\beta_p < 1$ for $p \geq 3$.

We conclude this section with a generalization of Theorem \ref{thEss}. 
This proof is based on Essouabri's method.
It is to be emphasized that the above results yield $\alpha_p \neq \beta_p$.

\begin{theorem}
Let $P \in \Integer[X,Y]$ be a $\T$-elliptic and $\T$-positive polynomial with its degree $d \geq 1$ and 
assume that $\coefficient{p}{P}{u}$ fails to converge as $u \to \infty$.
Then we find that $\Paszeta{p}{P}{1}{s}$ has a non-real pole on its axis of absolute convergence $\tset{\sigma = \theta_p}$. 
\end{theorem}

\begin{proof}
First, we remark that Landau's theorem (see \cite[Theorem 10]{Har-Rie} for example) implies that the point $s = \theta_p$ 
is a singularity of $\Paszeta{p}{P}{1}{s}$.
We should recall that Essouabri \cite{Ess05} mentioned that all singularities of $\Paszeta{p}{P}{1}{s}$ including $s = \theta_p$ 
must be simple poles.

Now we assume that $\Paszeta{p}{P}{1}{s}$ had no non-real poles on its axis of absolute convergence $\tset{\sigma = \theta_p}$.

Let $f(u) = \summatory{p}{P}{\ceil{u^{d/\theta_p}}}$ for $u>0$.
Then we find that 
\begin{equation*}
\Paszeta{p}{P}{1}{\theta_p s} = \int_{1^{-}}^{\infty} u^{-s} \, df(u).
\end{equation*}
Apply the Wiener--Ikehara theorem (see \cite[Chapter 5, Corollary 1]{Hla-Sch-Tas} for example) to obtain that $f(u)/u$ converges as $u \to \infty$.
This implies that $\summatory{p}{P}{\ceil{u}}/u^{\theta_p/d}$ converges as $u \to \infty$, 
but this is a contradiction to the assumption that $\coefficient{p}{P}{u}$ fails to converge.
\end{proof}

This theorem combined with Theorem \ref{thMainB} gives our goal Theorem \ref{thGoal}, 
so that it remains just to prove Theorem \ref{thMainA} and Theorem \ref{thMainB}.

\section{Proof of Theorem \ref{thMainA}} \label{secMainA}
We will prove Theorem \ref{thMainA} in this section and Theorem \ref{thMainB} in the next section.
In order to prove the former, we need the following lemma from Essouabri's article \cite{Ess05}. 

\begin{lemma} \label{lemEll}
Let $P \in \Real[X,Y]$ be a $\T$-elliptic polynomial with its degree $d \geq 1$.
Then there exist positive real numbers $c_1$, $c_2$ and $R$ such that $c_1 \norm{(x,y)}^d \leq P(x,y) \leq c_2 \norm{(x,y)}^d$ holds for
every $(x,y) \in \T(R)$.
\end{lemma}

For the proof, we refer to \cite[Lemma 2]{Ess05}.

Now we begin the proof of Theorem \ref{thMainA}.
First of all, we note that the desired estimate can be easily obtained by the above evaluations of $\alpha_p$ and $\beta_p$ 
in the previous articles \cite{Har77}, \cite{Sto77}, \cite{Ste89}, \cite{Wil98} when $P(X,Y) = X$, for example, 
\begin{equation} \label{3-C}
\frac{\beta_p}{2} u^{\theta_p} \leq \summatory{p}{X}{u} \leq 2 \alpha_p u^{\theta_p}
\end{equation}
for any sufficiently large $u \in \Natural$.

Let us consider the general case.
By Lemma \ref{lemEll}, we can find positive numbers $c_1$, $c_2$ and $R$ such that $c_1 \norm{(x,y)}^d \leq P(x,y) \leq c_2 \norm{(x,y)}^d$
for every $(x,y) \in \T(R)$. 
Since any $(x,y) \in \T$ satisfies $x \leq \norm{(x,y)} \leq 2x$, we have that 
\begin{equation} \label{3-A}
\frac{1}{2} \left( \frac{P(x, y)}{c_2} \right)^{1/d} \leq x \leq \left( \frac{P(x, y)}{c_1} \right)^{1/d}
\end{equation}
whenever $(x,y) \in \T(R)$.
By \eqref{3-A}, we find that all points $(x,y)$ on the curve segment $\tset{P(x,y) = u} \cap \T(R)$ satisfy 
\begin{equation} \label{3-B}
\frac{1}{2} \left( \frac{u}{c_2} \right)^{1/d} \leq x \leq \left( \frac{u}{c_1} \right)^{1/d}.
\end{equation}
If $u \in \Natural$ is sufficiently large, we find that the same curve segment must be contained in $\T(R)$ and \eqref{3-B} holds
for any $(x,y) \in \T$ with $P(x,y) = u$. 
This implies that 
\begin{equation*}
\summatory{p}{X}{\floor{\frac{1}{2} \left( \frac{u}{c_2} \right)^{1/d}}} 
\leq \summatory{p}{P}{u} 
\leq \summatory{p}{X}{\ceil{\left( \frac{u}{c_1} \right)^{1/d}}}
\end{equation*}
and we can conclude by \eqref{3-C} that, for example,
\begin{equation*}
\frac{\beta_p}{2 \cdot 4^{\theta_p} c_{2}^{\theta_p/d}} u^{\theta_p/d} 
\leq \summatory{p}{P}{u} 
\leq \frac{2^{\theta_p +1} \alpha_p}{c_{1}^{\theta_p/d}} u^{\theta_p/d}
\end{equation*}
as desired. 
(Note that $\floor{u} \geq u/2$ and $\ceil{u} \leq 2u$ when $u \geq 2$.)

\section{Proof of Theorem \ref{thMainB}} \label{secMainB}

Finally, we come to the stage of proving Theorem \ref{thMainB},
which claims that none of $\coefficient{p}{X+Y}{u}$ and $\coefficient{2}{X+2Y}{u}$ converge as $u \to \infty$.

The following proposition plays a key role in our proof.

\begin{proposition} \label{prLucas}
Suppose that $(m, n) \in \TLattice$ with its coordinates $m$ and $n$ expressed as $m = \sum_{j=0}^{h} m_j p^j$ and 
$n = \sum_{j=0}^{h} n_j p^j$ in the base $p$, respectively.
Then $(m, n)$ belongs to $\Pas{p}$ if and only if $m_j \geq n_j$ for $j = 0,1,\dots,h$, 
with the convention that $a<b$ implies $\binom{a}{b} = 0$.
\end{proposition}

\begin{proof}
This is an immediate corollary of Lucas' formula (introduced in his textbook \cite[Section 228]{Luc}): 
\begin{equation*}
\binom{m}{n} \equiv \prod_{j=0}^{h} \binom{m_j}{n_j} \pmod p.
\end{equation*}
(Note that $m_j$'s and $n_j$'s belong to $I_p$.)
\end{proof}

Let us begin the proof of the first part of Theorem \ref{thMainB}.
For simplicity, we abbreviate $\counting{p}{X+Y}{q}$, $\summatory{p}{X+Y}{u}$ and $\coefficient{p}{X+Y}{u}$ 
to $\ct{p}{q}$, $\sm{p}{u}$ and $\ce{p}{u}$, respectively.
In addition, we will extend the domain of $\phi_p$ to $\Integer$ for convenience: 
let $\ct{p}{q} = 0$ for $q<0$.
We are to prove several lemmas of which the first one is the essence of our proof.

\begin{lemma} \label{lemRec1}
Let $q$ be a non-negative integer and $r$ belong to $I_p$.
Then we have 
\begin{equation} \label{4-A}
\ct{p}{pq+r} = \floor{\frac{r}{2}+1} \ct{p}{q} + \floor{\frac{p-r}{2}} \ct{p}{q-1}.
\end{equation}
\end{lemma}

\begin{proof}
First, we remark that $\ct{p}{q}$ can be rewritten as 
$\ct{p}{q} = \# \tsetcond{n \in \Nonnegative}{(q-n, n) \in \Pas{p}}$.
In order to use Proposition \ref{prLucas}, we express $q$, $n$ and $q-n$ in the base $p$:
$q = \sum_{j=0}^{h} q_j p^j$, $n = \sum_{j=0}^{h} n_j p^j$ and $q-n = \sum_{j=0}^{h} m_j p^j$, respectively.
(We may assume that $q-n \geq 0$.)

Consider $v_a(q) = \# \tsetcond{n \in \Nonnegative}{n_0 = a, (q-n, n) \in \Pas{p}}$ for $a \in I_p$.
Then it is clear that $\ct{p}{q} = \sum_{a=0}^{p-1} v_a(q)$.

Actually we should calculate $v_a(q)$ to get the desired equation.
We note that
\begin{equation*}
m_0 = \begin{cases}
       q_0 - n_0 & (q_0 \geq n_0), \\
       q_0 - n_0 + p & (q_0 < n_0),
      \end{cases}
\end{equation*}
so that, by Proposition \ref{prLucas}, 
we find that $v_a(q) = 0$ if $m_0 < n_0 = a$, i.e. if $q_0/2 < a < q_0$ or $(q_0 + p)/2 < a \leq p-1$.
Otherwise, by the inequality $m_0 \geq n_0$ and Proposition \ref{prLucas}, we have that
\begin{align*}
v_a(q) &= \# \setcond{n \in \Nonnegative}{n_0 = a, m_j \geq n_j \, (j=1,2,\dots,h)} \\
       &= \# \setcond{\frac{n-a}{p} \in \Nonnegative}{\left( \frac{q-n-m_0}{p}, \frac{n-a}{p} \right) \in \Pas{p}} \\
       &= \begin{cases}
           \# \setcond{n' \in \Nonnegative}{\left( \frac{q-q_0}{p}-n', n' \right) \in \Pas{p}} & 
            \left( 0 \leq a \leq \frac{q_0}{2} \right), \\
           \# \setcond{n' \in \Nonnegative}{\left( \frac{q-q_0}{p}-1-n', n' \right) \in \Pas{p}} & 
            \left( q_0 \leq a \leq \frac{q_0+p}{2} \right) 
          \end{cases} \\
       &= \begin{cases}
           \ct{p}{\frac{q-q_0}{p}} & \left( 0 \leq a \leq \frac{q_0}{2} \right), \\
           \ct{p}{\frac{q-q_0}{p}-1} & \left( q_0 \leq a \leq \frac{q_0+p}{2} \right),
          \end{cases}
\end{align*}
where $n' = (n-a)/p$.

Hence we get 
\begin{align*}
\ct{p}{q} &= \sum_{a=0}^{\floor{q_0/2}} \ct{p}{\frac{q-q_0}{p}} 
           + \sum_{a=q_0+1}^{\floor{(q_0+p)/2}} \ct{p}{\frac{q-q_0}{p}-1} \\
          &= \floor{\frac{q_0}{2}+1} \ct{p}{\frac{q-q_0}{p}} + \floor{\frac{p-q_0}{2}} \ct{p}{\frac{q-q_0}{p}-1}.
\end{align*}
Replacing $q$ with $pq+r$, we obtain the conclusion.
\end{proof}

\begin{corollary} \label{corRec1}
Let $q$ be a positive integer and $k$ be a non-negative integer.
Then, 
\begin{equation} \label{4-B}
\ct{p}{p^k q-1} = \floor{\frac{p+1}{2}}^k \ct{p}{q-1}.
\end{equation}
\end{corollary}

\begin{proof}
We can easily verify it by induction on $k$ with \eqref{4-A} with noting that $p^k q -1 = p(p^{k-1}q-1)+(p-1)$.
\end{proof}

We now evaluate the value of $N_p$ from this corollary.

\begin{lemma} \label{lemNp1}
Let $u$ be a positive integer and $k$ be a non-negative integer.
Then we have that 
\begin{equation} \label{4-D}
\sm{p}{p^k u} = \begin{cases}
                 2^{k\theta_2} \sm{2}{u} - \frac{3^k -1}{2} \ct{2}{u-1} & (p=2), \\
                 p^{k\theta_p} \sm{p}{u} - \frac{p^k -1}{2} \cdot \left( \frac{p+1}{2} \right)^k \ct{p}{u-1} & (p \geq 3).
                \end{cases}
\end{equation}
\end{lemma}

\begin{proof}
By \eqref{4-A}, we can forward calculations as follows:
\begin{align*}
\sm{p}{pu} &= \sum_{q=0}^{u-1} \sum_{r=0}^{p-1} \ct{p}{pq+r} \\
           &= \sum_{q=0}^{u-1} \ct{p}{q} \sum_{r=0}^{p-1} \floor{\frac{r}{2}+1} 
            + \sum_{q=0}^{u-1} \ct{p}{q-1} \sum_{r=0}^{p-1} \floor{\frac{p-r}{2}} \\
           &= \sm{p}{u} \sum_{r=0}^{p-1} \left( \floor{\frac{r}{2}+1} + \floor{\frac{p-r}{2}} \right)
            - \ct{p}{u-1} \sum_{r=0}^{p-1} \floor{\frac{p-r}{2}}.
\end{align*}

Here we can check that $\sum_{r=0}^{p-1} (\floor{\frac{r}{2}+1} + \floor{\frac{p-r}{2}}) = p^{\theta_p}$ and
$\sum_{r=0}^{p-1} \floor{\frac{p-r}{2}} = B_p$, where
\begin{equation*}
B_p = \begin{cases}
       1 & (p=2), \\
       \frac{(p-1)(p+1)}{4} & (p \geq 3),
      \end{cases}
\end{equation*}
so that we get 
\begin{equation} \label{4-C}
\sm{p}{pu} = p^{\theta_p} \sm{p}{u} - B_p \ct{p}{u-1}.
\end{equation}

Hence we have that 
\begin{equation*}
\sm{p}{p^k u} = p^{k\theta_p} \sm{p}{u} - B_p \sum_{l=0}^{k-1} p^{l\theta_p} \ct{p}{p^{k-l-1}u-1}
\end{equation*}
by induction on $k$ with \eqref{4-C}.
Then we use \eqref{4-B} to finish the proof.
\end{proof}

Here appears a significant equation, which gives accumulation points of $\ce{p}{u}$.

\begin{lemma} \label{lemAc1}
Let $u$ be a positive integer.
Then we have that
\begin{equation} \label{4-E}
\lim_{k \to \infty} \ce{p}{p^k u} = \ce{p}{u} - \frac{\ct{p}{u-1}}{2u^{\theta_p}}.
\end{equation}
\end{lemma}

\begin{proof}
Dividing both sides of \eqref{4-D} by $(p^k u)^{\theta_p}$, 
we have the equation
\begin{equation*}
\ce{p}{p^k u} = \begin{cases}
                 \ce{2}{u} - \dfrac{1-3^{-k}}{2u^{\theta_2}} \ct{2}{u-1} & (p=2), \\
                 \ce{p}{u} - \dfrac{1-p^{-k}}{2u^{\theta_p}} \ct{p}{u-1} & (p \geq 3), 
                \end{cases}
\end{equation*}
which completes the proof.
\end{proof}

Now we complete our proof. 
Our aim is to find two distinct accumulation points of $\ce{p}{u}$.
Substituting $u=1$ into \eqref{4-E}, we have that $\lim_{k \to \infty} \ce{p}{p^k} = 1/2$.
Next we substitute $u=2$, and then we have $\lim_{k \to \infty} \ce{p}{2p^k} = 3/2^{\theta_p +1}$.
Since $\theta_p > \theta_2$ under the condition that $p \geq 3$, we obtain that $3/2^{\theta_p +1} < 3/2^{\theta_2 +1} = 1/2$,
and the second accumulation point $3/2^{\theta_p +1}$.

We need to make a rather sophisticated argument in case $p=2$.
Substituting $u=3$ into \eqref{4-E}, we have that $\lim_{k \to \infty} \ce{2}{3 \cdot 2^k} = 3^{1-\theta_2}$.
If $3^{1-\theta_2} = 1/2$, we would find that $\theta_2 (1-\theta_2) = -1$.
This is a contradiction to the Gel'fond--Schneider theorem (see \cite[Theorem 3.1--3.2]{Par-Sha} for example), 
which assures the transcendence of $\theta_2$,
hence $3^{1-\theta_2}$ must be the second accumulation point of $\ce{2}{u}$. 
Therefore the proof of the first part of Theorem \ref{thMainB} is completed.

\begin{remark}
Numerical calculations show that $\lim_{k \to \infty} \ce{2}{3 \cdot 2^k} = 3^{1-\theta_2} = 0.525898\dots$ and 
$\lim_{k \to \infty} \ce{2}{17 \cdot 2^k}= 0.487836\dots$.
Thus we expect that $\liminf_{u \to \infty} \ce{2}{u} < 1/2 < \limsup_{u \to \infty} \ce{2}{u}$.
\end{remark}

We are going to prove the second part of Theorem \ref{thMainB}. 
We can use the same method as that for the first part.
Thus we often omit details of the proofs of the following lemmas.
In order to prove the theorem, we only need to calculate $\counting{2}{X+2Y}{q}$, $\summatory{2}{X+2Y}{u}$ and $\coefficient{2}{X+2Y}{u}$,
but we note that some of the lemmas could be easily generalized to $\counting{p}{X+pY}{q}$ and so on.

Here we would change the meaning of our symbols; 
from now on, we abbreviate $\counting{p}{X+pY}{q}$, $\summatory{p}{X+pY}{u}$ and $\coefficient{p}{X+pY}{u}$ to
$\ct{p}{q}$, $\sm{p}{u}$ and $\ce{p}{u}$, respectively.
In addition, we again define the value of $\ct{p}{q}$ at every negative integer $q$ as $\ct{p}{q} = 0$.

First, we prove the essential formula.

\begin{lemma} \label{lemRecp}
Let $q$ be a non-negative integer and $r$ belong to $I_p$.
Then we have 
\begin{equation} \label{4-a}
\ct{p}{pq+r} = \sum_{a=0}^{r} \ct{p}{q-a}.
\end{equation}
\end{lemma}

\begin{proof}
The proof will proceed similarly to that of Lemma \ref{lemRec1}.

Since 
\begin{align*}
\ct{p}{q} &= \# \setcond{n \in \Nonnegative}{(q-pn, n) \in \Pas{p}} \\
          &= \sum_{a=0}^{p-1} \left( \# \setcond{n \in \Nonnegative}{n_0 = a, (q-pn, n) \in \Pas{p}} \right), 
\end{align*}
with $q = \sum_{i=0}^{h} q_i p^i$, $n = \sum_{i=0}^{h} n_i p^i$ and $q-n = \sum_{i=0}^{h} m_i p^i$ in the base $p$,
we should calculate $v_a(q) = \# \setcond{n \in \Nonnegative}{n_0 = a , (q-pn, n) \in \Pas{p}}$ for every $a \in I_p$
to obtain the desired equation.
Noting that $m_0 = q_0$, we can find that
\begin{align*}
v_a(q) &= \# \setcond{\frac{n-a}{p} \in \Nonnegative}{\left( \frac{q-pn-m_0}{p}, \frac{n-a}{p} \right) \in \Pas{p}} \\
       &= \# \setcond{n' \in \Nonnegative}{\left( \frac{q-q_0}{p}-a-pn', n' \right) \in \Pas{p}} \\
       &= \ct{p}{\frac{q-q_0}{p}-a}
\end{align*}
in case $q_0 \geq a$, where $n' = (n-a)/p$, and otherwise $v_a(q) = 0$.
We conclude the proof by replacing $q$ with $pq+r$.
\end{proof}

\begin{corollary} \label{corRec2}
Let $q$ be a positive integer and $k$ be a non-negative integer.
In addition, we set $\gamma_{+} = (1+\sqrt{5})/2$ and $\gamma_{-} = (1-\sqrt{5})/2 = -\gamma_{+}^{-1}$.
Then,
\begin{equation} \label{4-b}
\ct{2}{2^k q-1} 
= \frac{(\gamma_{+}^{k+1} - \gamma_{-}^{k+1}) \ct{2}{q-1} + (\gamma_{+}^{k} - \gamma_{-}^{k}) \ct{2}{q-2}}{\gamma_{+} - \gamma_{-}}.
\end{equation}
\end{corollary}

\begin{proof}
It follows from \eqref{4-a} the ``Fibonacci-like'' recurrence formula $\ct{2}{2^k q-1} = \ct{2}{2^{k-1}q-1} + \ct{2}{2^{k-2}q-1}$,
and then we can obtain the conclusion with easy calculations.
\end{proof}


What we have to do next is to sum up $\ct{p}{q}$'s.

\begin{lemma} \label{corNp2}
Let $u$ be a positive integer and $k$ be a non-negative integer.
Then,
\begin{equation*}
\sm{p}{p^k u} = p^{k\theta_p} \sm{p}{u} - \sum_{l=0}^{k-1} p^{l\theta_p} \sum_{b=1}^{p-1} \frac{(p-b)(p-b+1)}{2} \ct{p}{p^{k-l-1}u-b}.
\end{equation*}
In particular, we have that
\begin{equation} \label{4-d}
\sm{2}{2^k u} = 2^{k\theta_2} \sm{2}{u} - \sum_{l=0}^{k-1} 2^{l\theta_2} \ct{2}{2^{k-l-1}u-1}.
\end{equation}
\end{lemma}

\begin{proof}
First, we transform the summatory function $\sm{p}{pu}$.
We can forward calculations with \eqref{4-a} as follows: 
\begin{align*}
\sm{p}{pu} &= \sum_{q=0}^{u-1} \sum_{r=0}^{p-1} \sum_{a=0}^{r} \ct{p}{q-a} \\
           &= \sum_{a=0}^{p-1} (p-a) \sum_{q=0}^{u-1} \ct{p}{q-a} \\
           &= \sum_{a=0}^{p-1} (p-a) \left( \sm{p}{u} - \sum_{b=1}^{a} \ct{p}{u-b} \right) \\
           &= p^{\theta_p} \sm{p}{u} - \sum_{b=1}^{p-1} \frac{(p-b)(p-b+1)}{2} \ct{p}{u-b}.
\end{align*}

Then we obtain the desired equations by induction on $k$.
\end{proof}

We combine \eqref{4-b} with \eqref{4-d} to obtain the following decisive formula on $\ce{2}{u}$.

\begin{lemma} \label{lemAc2}
Let $u$ be a positive integer.
Then we have that
\begin{equation} \label{4-e}
\lim_{k \to \infty} \ce{2}{2^k u} = \ce{2}{u} - \frac{3 \ct{2}{u-1} + \ct{2}{u-2}}{5u^{\theta_2}}.
\end{equation}
\end{lemma}

\begin{proof}
By Cor.\ref{corRec2} and Cor.\ref{corNp2}, we have that
\begin{align*}
\ce{2}{2^k u} &= \ce{2}{u} - \frac{1}{(\gamma_{+} - \gamma_{-}) u^{\theta_2}} \times \\
&\times \sum_{l=0}^{k-1} 
\frac{(\gamma_{+}^{k-l} - \gamma_{-}^{k-l}) \ct{2}{u-1} + (\gamma_{+}^{k-l-1} - \gamma_{-}^{k-l-1}) \ct{2}{u-2}}{2^{(k-l)\theta_2}}.
\end{align*}
We immediately obtain the desired formula from this equation with noting that $2^{\theta_2} = 3 > \abs{\gamma_{\pm}}$.
\end{proof}

We now arrive at the place to complete our proof.
A good choice is to consider the case when $u=1$ and the case when $u=9$ in this situation.
Substituting each of them into \eqref{4-e}, 
we have that $\lim_{k \to \infty} \ce{2}{2^k} = 2/5$ and $\lim_{k \to \infty} \ce{2}{9 \cdot 2^k} = 64/(5 \cdot 9^{\theta_2})$.
If they were equal, we would have the equation $2{\theta}_2^2 -5 = 0$, 
which is a contradiction to the Gel'fond--Schneider theorem again.
Thus we could find two distinct accumulation points of $\ce{2}{u}$. 
Therefore we have just completed the proof of Theorem \ref{thMainB}.

\begin{remark}
Numerical calculations show that $64/(5 \cdot 9^{\theta_2}) = 0.393342\dots$, 
thus we expect that $\liminf_{u \to \infty} \ce{2}{u} < 2/5 \leq \limsup_{u \to \infty} \ce{2}{u}$.
%
\end{remark}

\begin{remark}
We would obtain the results for general $p$ if we could find the explicit forms of $\ct{p}{p^{k}u-b}$ for $b \in I_p$ like Corollary \ref{corRec2}. 
The sequences $\tset{\ct{p}{p^{k}u-b}}_{k = 0}^{\infty}$ for $b \in I_p$ are determined by a system of $p$ linear relations.
Although, in fact, it can be solved explicitly in the case when $p = 3$, the form of the solution is very complicated. 
\end{remark}
\begin{acknowledgment}
The author expresses a great gratitude to Prof.\ Kohji Matsumoto, Mr.\ Hiroki Fujino, Mr.\ Shinya Kadota and Mr.\ Yuta Suzuki 
for their helpful comments.
\end{acknowledgment}


IKKAI, Tomohiro\\
Graduate School of Mathematics, Nagoya University, Furocho, Chikusa-ku, Nagoya, 464-8602, Japan\\
E-mail address: m13006z@math.nagoya-u.ac.jp

\end{document}